\documentclass{article}

\usepackage{verbatim}
\usepackage[pdftex]{graphicx}
\usepackage{amsmath, amssymb, amsfonts, amsthm}
\usepackage{setspace}
\usepackage{algorithmic}
\usepackage{enumitem}
\usepackage{tabularx}
\usepackage{thmtools}
\usepackage{thm-restate}
\usepackage{sidecap}
\sidecaptionvpos{figure}{c}

\setlist{nolistsep}

\def \PFq[#1]{\mathbf{P}^#1(\mathbb{F}_q)}
\def \PF[#1,#2]{PG(#1,#2)}

\declaretheorem[name=Theorem]{theorem}

\newtheorem{lemma}[theorem]{Lemma}

\begin{document}
\title{A note on the largest induced matching in graphs avoiding a fixed bipartite graph}
\author{
Ben Lund\thanks{
		Department of Mathematics, Fine Hall, Princeton University, Princeton NJ 08544; {\sl lund.ben@gmail.com}. Research supported by NSF grant DMS-1802787.}
\and
Daniel Reichman\thanks{Department of Computer Science, Worcester Polytechnic Institute, Worcester, MA, 01609; {\sl daniel.reichman@gmail.com.}}	
}
\maketitle

\begin{abstract}
	We give a simple proof that every $n$-vertex graph $d$-regular graph that does not contain a fixed bipartite graph as a subgraph has an induced matching of size $\Omega((n/d)(\log d))$.
\end{abstract}

An \emph{induced matching} in an undirected graph $G$
is a matching $M$ where no two edges in $M$ are connected by a third edge. Induced matchings have received
attention in several contexts such as radio networks and parallelization capacity of neural systems.
One observation is that random or pseudorandom architectures that contain large induced matchings are useful
for interference-free processing~\cite{alon2017graph}. There has also been recent interest in approximation algorithms
for induced matchings~\cite{chalermsook2013graph}.

Mahdian showed~\cite{mahdian2000strong} that a graph with maximum degree $d$ that does not contain $C_4$ as a subgraph has strong chromatic index at most $(2+o(1))d^2/\ln d$.
The chromatic index is the minimum number of induced matchings that partition the edge set.
Since the number of edges in an $n$-vertex, $d$-regular graph is $nd/2$, Mahdian's result in particular implies that a $C_4$-free graph contains an induced matching of size $(n/(4-o(1))d) \ln d$.

Alon, Krivelevich, and Sudakov~\cite{alon1999coloring} gave a more general, but quantitatively weaker, result.
They show that the chromatic number of any graph with maximum degree $d$ in which the number of edges in the induced subgraph on the set of neighbors of any vertex does not exceed $d^2/f$ is at most $O(d/\log f)$.
The strong chromatic index of a graph $G$ is the chromatic number of the square of the line graph $L(G)^2$ of $G$.
Suppose that $G$ is an $n$-vertex, $d$-regular graph that does not have $C_4$ as a subgraph.
Note that the maximal degree of $L(G)^2$ is upper bounded by $2d^2$.
Every edge in the induced subgraph of $L(G)^2$ on the neighborhood of a vertex $v$ in $L(G)^2$ corresponds to an edge among $O(d^2)$ vertices of $G$.
Since $G$ is $C_4$-free, by the K\H{o}v\'ari-S\'os-Tur\'an theorem there are at most $O(d^{3})$ edges among these $O(d^2)$ vertices.
Hence, we may take $f=\Omega(d)$ in the theorem of Alon, Krivelevich, and Sudakov, which establishes that the chromatic number of $L(G)^2$ (and hence the strong chromatic index of $G$) is $O(d^2/\log d)$.

Here, we give a simple proof that any $d$-regular graph that avoids a fixed bipartite graph must have an induced matching of size $\Omega((n/d) \log d)$.
The regularity assumption is necessary as otherwise the size of a maximum matching might be $1$.
We rely on a standard result that a graph with few triangles has a large independent set \cite[Lemma 12.16]{bollobas2001random}.
For convenience, we sketch the proof of this result here.

\begin{lemma}~\label{thm:removal}
	Let $\epsilon$ be a constant in $(0,3)$ and suppose that $G$ has $n$ vertices, maximum degree $d$ and at most $nd^{2-\epsilon}$ triangles.
	Then $G$ contains an independent set of size $\Omega((n/d) \log d)$. Furthermore, there is a polynomial algorithm that finds such an independent set with probability $>2/3$.
\end{lemma}
\begin{proof}
	Let $a=\epsilon/3$. By our assumption, $G$ has at most $nd^{2-3a}$ triangles. Keep every vertex independently with probability $p=d^{-1+a}$ and delete it otherwise.
	The expected number of remaining vertices is $np$, the expected number of remaining edges is at most $ndp^2/2$ and the expected number of surviving triangles is at most $nd^{2-3a}p^3<np/40$ (assuming $d>d_0$ for sufficiently large constant $d_0$).
	Furthermore, by standard concentration results, with probability at least $9/10$ the number of remaining vertices is between $np/2$ and $3np/2$ and by Markov inequality and the choice of parameters with probability at least $9/10$ the number of surviving triangles is at most $np/4$. Deleting a single vertex from each triangle results with a graph $G'$ with at least $np/4$ vertices. Finally, by Markov, the number of edges in $G'$ is larger than $5ndp^2$ with probability at most $1/10$.
	It follows that the average degree of $G'$ is with probability at least $7/10>2/3$ at most $40dp=40d^{a}$. Therefore, using~\cite{shearer1983note, alon2004probabilistic}) $G'$ has an independent set of size $\Omega(np\log(dp)/(dp))=\Omega((n/d) \log d)$ and it can be found in polynomial time. 
\end{proof}
The algorithm above can be derandomized using a family $4$-wise independent random variables. We omit the details.

The claimed result on induced matchings follows by applying Lemma~\ref{thm:removal} to a sufficiently large, not necessarily induced, matching.

\begin{theorem}\label{thm:inducedmatching}
	Let $H$ be bipartite graph with $B>2$ vertices where $B$ is a constant independent of $n$.
	If $G$ is an $n$-vertex $d$-regular graph avoiding $H$ as a subgraph, then it must contain an induced matching of size $\Omega((n/d)\log d)$.
\end{theorem}
\begin{proof}
	It an easy consequence of Vizing's theorem that an $n$-vertex, $d$-regular graph must contain a matching of size at least $dn/(2d+2) \geq n/4$ ~\cite{yuster2013maximum}.
	Choose a matching $M$ of size $n/4$ and consider the subgraph $G'$ of $G$ induced by $V(M)$.
	Let $G_M$ be the graph obtained from $G'$ by contracting the edges of $M$.
	For any fixed vertex $v \in G_M$, every edge in the induced subgraph $G_v$ of $G_M$ on the neighborhood of $v$ corresponds to an edge in $G$ among one of at most $4d$ vertices.
	Since $G$ does not contain a $K_{B,B}$, the K\H{o}v\'ari-S\'os-Tur\'an theorem implies that $G_v$ has at most $O(d^{2-1/B})$ edges.
	This in turn implies that $G_M$ contains $O(nd^{2-1/B})$ triangles.
	The result now follows from Lemma~\ref{thm:removal}.
\end{proof}

{\bf Acknowledgments:} We thank Luke Postle and Ross Kang for informing us of Mahdian's result. We are very grateful to Benny Sudakov for useful discussions and informing us of Lemma~\ref{thm:removal}.

\bibliographystyle{alpha}
\bibliography{IM}

\end{document}